\documentclass[reqno,tbtags,intlimits,a4paper,oneside,11pt]{amsart}
\usepackage{amsfonts,amsmath,amssymb}
\newtheorem{Theorem}{Theorem}[section]

\def\beq#1#2\eeq{%
        \begin{equation}%
        \label{#1}%
            #2%
        \end{equation}%
    }

\usepackage{color}
\usepackage{graphicx}
\usepackage{epstopdf}

\title[Bernoulli-Hirzebruch]{Todd polynomials and Hirzebruch numbers}

\author{V.M. Buchstaber}\address{Steklov Mathematical Institute and Moscow State University, Russia}
\email{buchstab@mi-ras.ru}

\author{A.P. Veselov}
\address{Department of Mathematical Sciences,
Loughborough University, Loughborough LE11 3TU, UK}
\email{A.P.Veselov@lboro.ac.uk}

\begin{document}

\maketitle

{\begin{center}\normalsize\it
     To our teacher Sergei Petrovich Novikov on his 85-th birthday
   \end{center}}

\begin{abstract}
In 1956 Hirzebruch found an explicit formula for the denominators of the Todd polynomials, which was proved later in his joint work with Atiyah.
We present a new formula for the Todd polynomials in terms of the ``forgotten symmetric functions", which follows from our previous work on complex cobordisms. In particular, this leads to a simpler proof of the Hirzebruch formula and provides new interpretations for the Hirzebruch numbers.
\end{abstract}


\section{Introduction}

Todd polynomials play a seminal role in algebraic geometry and topology (see \cite{Hirz-66, B-2012}). In terms of the Chern classes they were introduced by Hirzebruch \cite{Hirz-56} via the generating function
\beq{bern}
\prod_{i=1}^n\frac{tx_i}{1-e^{-tx_i}}=\sum_{k=0}^\infty T_k(c_1,\dots, c_n) t^k,
\eeq
where the Chern classes $c_k=e_k$ are the elementary symmetric functions of $x_1,\dots, x_n$ known as Chern roots (see \cite{Hirz-66}):
$$
T_0=1, T_1=\frac{1}{2} c_1, \quad
T_2=\frac{1}{12} (c_1^2+c_2), \quad
T_3=\frac{1}{24} c_1 c_2,
$$
$$
T_4=\frac{1}{720} (-c_1^4+4c_1^2c_2+3c_2^2+c_1c_3-c_4), \,\,
T_5=\frac{1}{1440} (-c_1^3c_2+3c_1c_2^2+c_1^2c_3-c_1c_4),
$$
$$
T_6=\frac{1}{60480} (2c_1^6-12c_1^4c_2 +11c_1^2c_2^2+10c_2^3+5c_1^3c_3+11c_1c_2c_3-c_3^2
$$
$$-5c_1^2c_5-9c_2c_4-2c_1c_5+2c_6).
$$
They appear both in the celebrated Hirzebruch-Riemann-Roch theorem and Atiyah-Singer Index formula \cite{Hirz-66}.
Since the corresponding Todd number $T_n(M^{2n})$ is integer for any stably complex manifold $M^{2n}$ (see \cite{Mil-60, Nov-60}), this leads to the important divisibility conditions \cite{Hirz-58, Hirz-66} for the characteristic Chern numbers of such manifolds by the denominator of $T_n$. Since the Todd number of any complex projective space $\mathbb CP^n$ equals 1, these conditions cannot be improved.

Hirzebruch \cite{Hirz-56} discovered the following formula for the denominators $\mu(T_k)$ of the Todd polynomials $T_k$:
\beq{Hirz1}
\mu(T_k)=\prod_{p \, prime}p^{\lfloor \frac{k}{p-1}\rfloor},
\eeq
where $\lfloor x \rfloor$ means the integer part of $x$ (largest integer not exceeding $x$). The proof of this formula was given later in his joint work with Atiyah \cite{AH}.

We will denote the numbers given by the right hand side of (\ref{Hirz1}) as $\mathfrak h_k$ and call them {\it Hirzebruch numbers}:
$$
\mathfrak h_0=1, \,\, \mathfrak h_1=2,\,\, \mathfrak h_2=12,\,\, \mathfrak h_3=24,\,\, \mathfrak h_4=720,\,\, \mathfrak h_5=1440,\,\, \mathfrak h_6=60480.
$$
They satisfy the property
$
\mathfrak h_{2k+1}=2\mathfrak h_{2k}
$
and appear in the On-line Encyclopedia of Integer Sequences (in a different interpretation) as sequence A091137. 
 
In this paper we present new interpretations of these numbers.

\begin{Theorem}
The Hirzebruch number $\mathfrak h_k$ is the least common multiple of the numbers
\beq{LCM1}
\mathfrak h_k=lcm((\lambda_1+1)!\dots (\lambda_l+1)!, \quad \lambda=(\lambda_1,\dots,\lambda_l) \in \mathcal P_k)
\eeq
or, equivalently,
\beq{LCM2}
\mathfrak h_k=lcm((\lambda_1+1)\dots (\lambda_l+1), \quad \lambda=(\lambda_1,\dots,\lambda_l) \in \mathcal P_k),
\eeq
where $\mathcal P_k$ is the set of all partitions $\lambda=(\lambda_1,\dots, \lambda_l), |\lambda|:=\lambda_1+\dots +\lambda_l=k.$
\end{Theorem}


The central result of our paper is the following new formulas for the Todd polynomials (or, more precisely, for the corresponding Todd symmetric functions) in terms of the ``forgotten symmetric functions" $f_\lambda$ \cite{Mac, Stanley}.

\begin{Theorem}
The Todd symmetric function $T_k$ can be expressed in terms of the forgotten symmetric functions as follows 
\beq{Toddforint}
T_k=\sum_{\lambda: |\lambda|=k}\frac{f_\lambda}{(\lambda_1+1)!\dots (\lambda_l+1)!}.
\eeq
\end{Theorem}

We introduce also a new basis $g_\lambda$ of symmetric functions over $\mathbb Z$ and show that $T_k$ can be expressed without factorials in the denominators:
\beq{Toddgint}
T_k= (-1)^k\sum_{\lambda: |\lambda|=k}\frac{g_\lambda}{(\lambda_1+1)\dots (\lambda_l+1)}.
\eeq
We have also similar new formula for the signature $\tau(M^{4k})$ of any oriented $4k$-dimensional manifold $M^{4k}$ (see formula (\ref{sig2}) below).

The proofs are based on the results from the paper \cite{BV-2020}, where we found an explicit formula for the exponential of the formal group of geometric complex cobordisms \cite{Nov-67, B-1970}.

As a corollary we have the equivalence of two formulas in Theorem 1.1 and a new proof of the Hirzebruch formula.

We discuss also the versions of Hirzebruch numbers, which appear in relation with Hirzebruch signature formula and Steenrod's cycle realisation problem. In the last section we derive two formulas for the Bernoulli numbers as the sums over partitions and discuss the relation with the classical von Staudt-Clausen result.

%

\section{Chern classes and symmetric functions}

For the theory of characteristic classes we refer to the book of Milnor and Stasheff \cite{MS}, for the theory of symmetric functions - to Macdonald \cite{Mac} and Stanley \cite{Stanley}.

Let $\xi$ be a complex vector bundle of rank $n$ over some manifold $M$, then
the Chern classes $c_k\in H^{2k}(M, \mathbb Z)$ can be viewed as the elementary symmetric functions of the Chern roots $x_1,\dots, x_n$ defined by the relation
\beq{ct}
c(t):=1+c_1 t+ c_2 t^2+\dots +c_n t^n=(1+x_1 t)\dots (1+x_n t).
\eeq
It is convenient here to get rid of the explicit dependence on $n$ by replacing the algebra $\Lambda_n$ of symmetric polynomials of variables $x_1,\dots, x_n$ by the algebra of {\it symmetric functions} $\Lambda$ defined as
the inverse limit of $\Lambda_n$ in the category of graded algebras
$$\Lambda=\lim_{\longleftarrow} \Lambda_{n}=\oplus_{k\geq 0}\Lambda^k.$$
By definition, symmetric function $f \in \Lambda^k$ of degree $k$ corresponds to an infinite sequence of elements $f_n \in \Lambda^k_n, \, n=1,2,\dots$ of degree $k$ such that
$$\varphi_{m,n} f_m = f_n,$$ where for any $m>n$ the homomorphisms
$\varphi_{m,n}: \Lambda_m \rightarrow \Lambda_n$ send $x_i$ with $i>n$ to zero. There is a canonical surjective homomorphism
\beq{can}
\varphi_n: \Lambda \to \Lambda_n
\eeq
sending all $x_i$ with $i>n$ to zero.

The algebra $\Lambda$ (considered over $\mathbb Z$) has several convenient bases of symmetric functions labeled by the partitions $\lambda=(\lambda_1,\dots, \lambda_l):$

\noindent {\it elementary symmetric functions} $$e_\lambda=e_{\lambda_1}\dots e_{\lambda_l}, \quad e_k=\sum_{i_1<\dots<i_k}x_{i_1}\dots x_{i_k},$$
{\it complete symmetric functions} $$h_\lambda=h_{\lambda_1}\dots h_{\lambda_l}, \quad h_k=\sum_{i_1\leq \dots\leq i_k}x_{i_1}\dots x_{i_k},$$
{\it monomial symmetric functions} $$m_\lambda=\sum_\alpha x_1^{\alpha_1}x_2^{\alpha_2}\dots x_l^{\alpha_l},$$
where the sum is taken over all {\it different} permutations $\alpha=(\alpha_1,\dots,\alpha_l)$ of  $\lambda=(\lambda_1,\dots, \lambda_l).$
(Note that the power sums $p_k = x_1^k+x_2^k + \dots,\,\, k=1,2, \dots$ generate $\Lambda$ over $\mathbb Q$, but not over $\mathbb Z$ since e.g. $e_2=\frac{1}{2}(p_1^2-p_2)$.)
 
 There is also less known basis of the so-called {\it ``forgotten" symmetric functions} 
 \beq{forg} 
 f_\lambda=\omega(m_\lambda), \quad n=|\lambda|:=\lambda_1+\dots +\lambda_l,
 \eeq
 where $\omega$ is the standard involution of $\Lambda$ uniquely defined by the property $\omega(e_k)=h_k$  (see Macdonald \cite{Mac}, Section I.2 ).
 (Note that Stanley \cite{Stanley} defines $f_\lambda$ with additional sign $(-1)^{n-l}$; we will follow Macdonald's convention).
 Lenart \cite{Lenart} used these symmetric functions in connection with Lazard's theorem in the theory of formal groups.
 
 To explain the nature of the involution $\omega$ consider the corresponding generating functions
 \beq{eh}
 E(t):=\sum_{k\geq 0}e_kt^k=\prod_{i\geq 1}(1+x_i t), \quad H(t):=\sum_{k\geq 0}h_kt^k=\prod_{i\geq 1}(1-x_i t)^{-1},
 \eeq
and note that they satisfy the duality relation
\beq{dual}
E(-t)H(t)\equiv 1.
\eeq
The generating function of the symmetric power-sum functions is
$$
P(t)=\sum_{k\geq 0}p_{k+1}t^k=\sum_{i\geq 1} \frac{x_i}{1-x_it}=\frac{d}{dt}\log \prod_{i\geq 1}(1-x_i t)^{-1}=\frac{H'(t)}{H(t)},
$$
so on the power sums $\omega$  acts by
$$
\omega(p_k)=(-1)^{k-1}p_k.
$$

The forgotten symmetric functions can be expressed via monomial symmetric function by the following combinatorial formula
\beq{comb}
(-1)^{n-l}f_\lambda=\sum_{\mu}a_{\lambda \mu}m_\mu,
\eeq
where $n=|\lambda|$ as before, $l=l(\lambda)$ is the number of parts of $\lambda=(\lambda_1,\dots,\lambda_l)$ and  $a_{\lambda \mu}$ equals the number of different permutations $\alpha=(\alpha_1,\dots,\alpha_l)$ of  $\lambda=(\lambda_1,\dots, \lambda_l)$, such that
\beq{cond}
 \{\mu_1+\dots + \mu_j: \,\, 1\leq j \leq l(\mu)\} \subseteq \{\alpha_1+\dots + \alpha_i: \, 1\leq i \leq l\}
\eeq
(see Stanley \cite{Stanley}, Ex.7.9). In particular, when $|\lambda|=2$ we have two partitions: $(2)$ and $(1,1)$ and two forgotten symmetric functions
$$
f_{(1,1)}=m_{(1,1)}+m_{(2)}, \quad f_{(2)}=-m_{(2)}
$$
(the formula $f_{(k)}=(-1)^{k-1} m_{(k)}$ is valid for all $k$). When $|\lambda|=3$ we have three forgotten symmetric functions
$$
f_{(1,1,1)}=m_{(1,1,1)}+m_{(2,1)}+m_{(3)}, \quad f_{(2,1)}=-m_{(2,1)}-2m_{(3)}, \quad f_{(3)}=m_{(3)}.
$$
The following result explains their relation with the Todd polynomials. 

More precisely, we define the {\it Todd symmetric functions} $T_k\in \Lambda^k \subset \Lambda$ uniquely by the property that the corresponding image $\varphi_k(T_k)\in \Lambda_k$ coincide with the Todd polynomials $T_k(c_1,\dots,c_k)$. Then Theorem 1.2 states that 
\beq{Toddfor2}
T_k=\sum_{\lambda: |\lambda|=k}\frac{f_\lambda}{(\lambda+1)!},
\eeq
where we used the notation $(\lambda+1)!:=(\lambda_1+1)!\dots (\lambda_l+1)!.$


{\bf Proof of Theorem 1.2.}
We use the following formula from our paper \cite{BV-2020} for the cobordism class for any $U$-manifold $M^{2n}$
\beq{our1}
[M^{2n}]=\sum_{\lambda: |\lambda|=n}c^{\nu}_\lambda(M^{2n})\frac{[\Theta^\lambda]}{(\lambda+1)!},
\eeq
where $\Theta^k$ is a smooth theta divisor of a general principally polarised abelian variety $A^{k+1}$ and $\Theta^\lambda=\Theta^{\lambda_1}\times \dots \times \Theta^{\lambda_l}.$
Here $c^{\nu}_\lambda(M^{2n})$ are the Chern numbers of the {\it normal} bundle of $M^{2n}$, corresponding to {\it monomial} symmetric functions $m_\lambda$ (see formula (10) in \cite{BV-2020}).

Since the Todd genus $Td(\Theta^k)=(-1)^k$ (see \cite{BV-2020}) the Todd genus \cite{Todd} for any $U$-manifold $M^{2n}$ can be expressed as
\beq{Todd}
Td(M^{2n})=(-1)^n\sum_{\lambda: |\lambda|=n}\frac{c^{\nu}_\lambda(M^{2n})}{(\lambda+1)!},
\eeq
Note that the Chern classes $c_k^\nu\in H^{2k}(M^{2n}, \mathbb Z)$ of the normal bundle $\nu(M^{2n})$ and usual Chern classes $c_l$ of the tangent bundle $TM^{2n}$ are related by
\beq{norm}
(1+c_1t+\dots+c_nt^n)(1+c_1^\nu t+\dots+c_m^{\nu}t^m)\equiv 1.
\eeq
Comparing this with the duality relation (\ref{dual}) written as $E(t)H(-t)\equiv 1$ we see that the Chern numbers of tangent and normal bundles (up to a sign) are related by involution $\omega$. Now since $f_\lambda=\omega(m_\lambda)$ we have
\beq{normfor}
c^{\nu}_\lambda(M^{2n})=(-1)^n \varphi_n(f_\lambda),
\eeq
which implies formula (\ref{Toddfor}).

In particular, we have 
$$
T_2=\frac{f_{(1,1)}}{2!\cdot 2!}+\frac{f_{(2)}}{3!}=\frac{1}{4}(m_{(1,1)}+m_{(2)})-\frac{1}{6} m_{(2)}=\frac{1}{12}(e_1^2+e_2),
$$
$$
T_3=\frac{f_{(1,1,1)}}{2!\cdot 2!\cdot 2!}+\frac{f_{(2,1)}}{3!\cdot 2!}+\frac{f_{(3)}}{4!}=\frac{m_{(1,1,1)}+m_{(2,1)}+m_{(3)}}{8}-\frac{2m_{(3)}+m_{(2,1)}}{12}+\frac{m_3}{24}$$
$$=\frac{1}{24}(m_{(2,1)}+3m_{(3)})=\frac{1}{24}e_1e_2
$$
in agreement with Hirzebruch.

\section{New interpretations of Hirzebruch numbers}

Note first that from Theorem 2.1 we have a new interpretation of the Hirzebruch numbers.

\begin{Theorem}
The denominator $\mathfrak h_k$ of the Todd polynomial $T_k$ is the least common multiple of the numbers $(\lambda+1)!, \,\,  \lambda \in \mathcal P_k$:
\beq{lcm1}
\mathfrak h_k=lcm((\lambda_1+1)!\dots (\lambda_l+1)!, \quad \lambda=(\lambda_1,\dots,\lambda_l) \in \mathcal P_k).
\eeq
\end{Theorem}

We are going to show now that $(\lambda_i+1)!$ can be replaced here simply by $(\lambda_i+1)$, which will lead to a new proof of the Hirzebruch formula (\ref{Hirz1}).

For this we use the results of our paper \cite{BV-2020} where it was in particular shown that the exponential of the formal group in complex cobordisms can be given by
the formula
\beq{CDnew}
\beta(v)=v+\sum_{n=1}^\infty[\Theta^n]\frac{v^{n+1}}{(n+1)!},
\eeq
where $\Theta^n$ is a smooth theta divisor of a general principally polarised abelian variety $A^{n+1}$, considered as the complex manifold of real dimension $2n.$
The inverse of this series is the logarithm of this formal group, which can be given explicitly by
the Mischenko series  \cite{Nov-67, BMN-71}:
\beq{Mis}
\alpha(u)=\beta^{-1}(u)=u+\sum_{n=1}^\infty[\mathbb CP^n]\frac{u^{n+1}}{n+1}.
\eeq

Denote $$a_n=\frac{[\mathbb CP^n]}{n+1}, \, b_n=\frac{[\Theta^n]}{(n+1)!}, \,$$
then we have two power series 
$$\alpha(u)=u+\sum_{n\geq 1}a_n u^{n+1}, \quad \beta(v)=v+\sum_{n\geq 1}b_n v^{n+1},$$
which are inverse to each other under substitution: $$\alpha(\beta(v))\equiv v, \,\, \beta(\alpha(u))\equiv u.$$

The coefficients of such series are known to be related by the inversion formulae, 
expressing each other as certain polynomials with integer coefficients:
$$
b_1=-a_1, \quad b_2=-a_2+2a_1^2, \quad b_3=-a_3+5a_1a_2-5a_1^3,$$
$$
b_4=-a_4+6a_1a_3+3a_2^2-21a_1^2a_2+14a_1^4.
$$
Remarkably these polynomials can be interpreted in terms of the combinatorics of the {\it  associahedra} (Stasheff polytopes), see \cite{Loday} for the details.
For example, the formula for $b_4$ says that $3D$ associahedron has 6 pentagonal and 3 quadrilateral faces, 21 edges and 14 vertices.
In particular, the last coefficient of $b_n$ (up to a sign) is always the {\it Catalan number} $C_n=\frac{1}{n+1}{ 2n \choose n}$, which equals the number of vertices of the corresponding associahedron. A related geometric interpretation in terms of strata in the Deligne-Mumford compactification of the moduli space $M_{0,n}$ is given by McMullen \cite{McMullen}.

These polynomials can also be expressed in terms of the {\it partial (exponential) Bell polynomials} $B_{n,k}$ (see e.g. \cite{Comtet}, Section 3.8)),
or, more explicitly, as
\beq{invb}
b_n=\frac{1}{(n+1)!} \sum_{k_1+2k_2+\dots+nk_n=n, \, k_i\geq 0} (-1)^k\frac{(n+k)!} {k_1!k_2!\dots k_n!}a_1^{k_1}a_2^{k_2}\dots a_n^{k_n},
\eeq
where $k=k_1+k_2+\dots+k_n$ (see e.g. \cite{Gessel}, formula (2.5.1)).

These formulas allow to write any monomial $b^\lambda=b_1^{m_1}\dots b_k^{m_k},$ where $m_j$ is the number of appearances of $j$ in $\lambda,$
as a linear combination of the corresponding $a^\mu$ with integer coefficients $C^\lambda_\mu$:
\beq{clm}
b^\lambda=\sum_{\mu} C^\lambda_\mu a^\mu.
\eeq
The matrix $C^\lambda_\mu$ is triangular (in lexicographic order) with diagonal elements $C^\lambda_\lambda=(-1)^{l(\lambda)}.$
Substituting this into (\ref{our1}) we can express any cobordism class as a polynomial 
\beq{our2}
[M^{2n}]=\sum_{\mu: |\mu|=n}A_\mu(M^{2n})a^\mu, \quad a_k=\frac{[\mathbb CP^k]}{k+1}
\eeq
with integer $A_\mu(M^{2n}):=\sum_{\lambda} C^\lambda_\mu c^{\nu}_\lambda(M^{2n}).$
Taking into account that Todd genus of $\mathbb CP^k$ is 1, we have a new formula for the Todd genus
\beq{Todd2}
Td(M^{2n})=\sum_{\mu: |\mu|=n}\frac{A_\mu(M^{2n})}{(\mu_1+1)\dots(\mu_l+1)}.
\eeq

Introduce now a new basis of symmetric functions over $\mathbb Z$ by
\beq{g}
g_\lambda=\sum_{\mu} C_\lambda^\mu f_\mu,
\eeq
where $f_\mu,$ as before, the forgotten symmetric functions. 

In particular, for $|\lambda|=2$ we have
\beq{g2}
g_{(1,1)}=f_{(1,1)}+2f_{(2)}=m_{(1,1)}-m_{(2)}, \quad g_{(2)}=-f_{(2)}=m_{(2)},
\eeq
for $|\lambda|=3$
\beq{g3}
g_{(1,1,1)}=-(f_{(1,1,1)}+2f_{(2,1)}+5f_{(3)})=-m_{(1,1,1)}+m_{(2,1)}-2m_{(3)},
\eeq
$$
g_{(2,1)}=f_{(2,1)}+5f_{(3)}=-m_{(2,1)}+3m_{(3)}, \quad g_{(3)}=-f_{(3)}=-m_{(3)},
$$
for $|\lambda|=4$
\beq{g4}
g_{(1,1,1,1)}=f_{(1,1,1,1)}+2f_{(2,1,1)}+4f_{(2,2)}+5f_{(3,1)}+14f_{(4)},
\eeq
$$
g_{(2,1,1)}=-(f_{(2,1,1)}+4f_{(2,2)}+5f_{(3,1)}+21f_{(4)}),\,\,\quad g_{(2,2)}=f_{(2,2)}+3f_{(4)},
$$
$$
g_{(3,1)}=f_{(3,1)}+6f_{(3)},\,\, g_{(4)}=-f_{(4)}.
$$
We are not aware of any appearances of these functions before, but they seem to be worthy of further study.

Combining all this we have a new formula for the Todd symmetric functions.

\begin{Theorem}
The Todd symmetric function $T_k$ can be expressed in terms of the new symmetric functions $g_\lambda$ as follows 
\beq{Toddnew}
T_k=(-1)^k\sum_{\lambda: |\lambda|=k}\frac{g_\lambda}{(\lambda_1+1)\dots (\lambda_l+1)}.
\eeq
\end{Theorem}

As a corollary we can drop factorials in our formula (\ref{lcm1}). \footnote{Alexey Ustinov provided us with a direct number-theoretic proof of this result.}

\begin{Theorem}
The denominator $\mathfrak h_k$ of the Todd polynomial $T_k$ is the least common multiple 
\beq{lcm2}
\mathfrak h_k=lcm((\lambda_1+1)\dots (\lambda_l+1), \quad \lambda=(\lambda_1,\dots,\lambda_l) \in \mathcal P_k).
\eeq
\end{Theorem}

Now we can give a new proof of the Hirzebruch formula, proved in Atiyah-Hirzebruch \cite{AH}.

\begin{Theorem} (Hirzebruch)
The denominator $\mathfrak h_k$ of the Todd polynomial $T_k$ can be given as the product over prime numbers 
\beq{Hirz3}
\mathfrak h_k=\prod_{p \, prime}p^{\lfloor \frac{k}{p-1}\rfloor}.
\eeq
\end{Theorem}

Indeed, it is easy to see that  the integer part $m=\lfloor \frac{k}{p-1}\rfloor$ is the maximal power of a prime $p$ in the right hand side of the formula (\ref{lcm2}), 
coming from the partitions $\lambda=((p-1)^m, \dots),$ where $(p-1)^m$ means that $p-1$ is taken $m$ times and dots stand for the parts less than $p-1$.

\section{Other appearances of Hirzebruch numbers}

\subsection{Hirzebruch signature formula and $L$-polynomials}

One of the first famous results of Hirzebruch was the formula for the signature $L_k=L_k(M^{4k})$ of a real oriented manifold $M^{4k}$ in terms of the Pontrjagin classes $p_j$ \cite{Hirz-66}:
$$
L_0=1, \quad L_1=\frac{1}{3} p_1, \quad
L_2=\frac{1}{45} (7p_2-p_1^2), \quad
L_3=\frac{1}{945} (62 p_3-13 p_1p_2+2p_1^3),
$$
$$
L_4=\frac{1}{14175} (381p_4-71 p_1p_3-19p_2^2+22p_1^2p_2-3p_1^4).
$$
The generating function of these polynomials can be given by the Hirzebruch formula
\beq{signat}
\prod_{i=1}^n\frac{tx_i}{\tanh tx_i}=\sum_{k=0}^\infty L_k(p_1,\dots, p_n) t^{2k},
\eeq
where the Pontrjagin classes $p_k$ are the elementary symmetric functions $e_k$ of $x_1^2,\dots, x_n^2$.

\begin{Theorem} (Hirzebruch)
The denominator $\mu(L_k)$ of the $L$-polynomial $L_k$ can be given as the product over prime numbers 
\beq{Hirz_L}
\mu(L_k)=\prod_{p \, prime, \, p\geq 3}p^{\lfloor \frac{2k}{p-1}\rfloor}
\eeq
and is related to the corresponding Hirzebruch number by the formula
 \beq{Hirz_L2}
\mu(L_k)=2^{-2k} \mathfrak h_{2k}=2^{-2k-1}\mathfrak h_{2k+1}.
\eeq
\end{Theorem}

The proof is also given in Atiyah-Hirzebruch \cite{AH}. We provide now a simpler proof based on our formula (\ref{our2}).

\begin{Theorem} 
The denominator $\mu(L_k)$ of the $L$-polynomial $L_k$ is the least common multiple of the numbers
\beq{lcm3}
\mu(L_k)=lcm((2\lambda_1+1)\dots (2\lambda_l+1), \quad \lambda=(\lambda_1,\dots,\lambda_l) \in \mathcal P_k).
\eeq
\end{Theorem}

\begin{proof}
We use the fact that the natural homomorphism of complex cobordisms $\Omega_U \to \Omega_{SO}$ to the real oriented cobordisms is surjective modulo torsion (see Milnor \cite{Mil-60} and Novikov \cite{Nov-60, Nov-62}). The kernel of this homomorphism is the ideal in $\Omega_U$ generated by the cobordism classes of the manifolds of dimensions $4k+2, \, k=0, 1,\dots.$
This means that one can choose a representative of cobordism class $[M^{4n}]\in \Omega_{SO}$ module torsion among the $U$-manifolds.
Since the signature $\tau(M^{4n})$ is cobordism-invariant, well-defined modulo torsion and $\tau(M^{4k+2})=0$ we can use our formula (\ref{our2}) to compute it. Taking into account that the signatures $\tau(\mathbb CP^{2k})=1, \, \tau(\mathbb CP^{2k+1})=0$, we have
\beq{sig2}
\tau(M^{4n})=\sum_{\mu: |\mu|=n}\frac{A_\mu(M^{4n})}{(2\mu_1+1)\dots(2\mu_l+1)}.
\eeq
The claim now follows in the same way as in Todd case.
\end{proof}

It is easy to see that this  least common multiple equals the product in the right hand side of Hirzebruch formula (\ref{Hirz_L}), so as a corollary we have another proof of this formula.

\subsection{Hirzebruch numbers in Steenrod's cycle realisation problem}

The following result was found by Buchstaber in \cite{B-1969} in relation with the cycle realisation problem.

Let $X$ be a cell complex and $x\in H_n(X, \mathbb Z)$ be a cycle. Steenrod asked when $x$ can be realised as an image of some manifold, so there exists a compact oriented smooth manifold $N^n$ and a map $f: N^n \to X$ such that $f_*([N^n])=x.$

From the results of Thom \cite{Thom} it follows that up to some multiple, depending only on $n$, this is always possible. The question is what is the minimal multiple $\nu_{n}$ required such that $\nu_{n} x$ can be realised by a submanifold.

Consider the following {\it Buchstaber numbers}, which can be considered as a modification of the Hirzebruch numbers:
\beq{buch}
\mathfrak b_n=\prod_{p \, prime, \, p\geq 3}p^{\lfloor \frac{n-2}{2(p-1)}\rfloor}.
\eeq
$$
\mathfrak b_1=\mathfrak b_2=\mathfrak b_3=\mathfrak b_4=1, \quad \mathfrak b_5=\mathfrak b_6=\mathfrak b_7=\mathfrak b_8=3,
$$
$$
\mathfrak b_9=\mathfrak b_{10}=\mathfrak b_{11}=\mathfrak b_{12}=45,\,\, \mathfrak b_{13}=\mathfrak b_{14}=\mathfrak b_{15}=\mathfrak b_{16}=3^3\times 5\times 7=945.
$$
They are related to the Hirzebruch numbers by the formula
\beq{buch2}
\mathfrak b_{4k+1}=\mathfrak b_{4k+2}=\mathfrak b_{4k+3}=\mathfrak b_{4k+4}=2^{-2k}\mathfrak h_{2k}=\mu(L_k).
\eeq

\begin{Theorem} (Buchstaber \cite{B-1969})
For every cycle $x\in H_n(X, \mathbb Z)$ the cycle $\mathfrak b_n x$ can be realised by a manifold.
\end{Theorem}

Recently Rozhdestvenskii \cite{Rozh} improved this result by showing that the same is true if we replace $n-2$ in the formula for $\mathfrak b_n$ by $n-3$, but the minimal value of multiple in Steenrod's problem for general $n$ is still unknown.

\section{Formulas for the Bernoulli numbers and von Staudt-Clausen theorem}

Note that the Todd polynomials are essentially {\it Bernoulli polynomials of higher order}, which were introduced by N\"orlund \cite{Nor} by the generating function
\beq{bernoul}
\prod_{i=1}^n\frac{tx_i}{e^{tx_i}-1}=\sum_{k=0}^\infty B_k^{(n)}(x_1,\dots, x_n) \frac{t^k}{k!}.
\eeq
Namely, assuming that $c_i,$ as before, are elementary symmetric functions $e_i$ of $x_1,\dots, x_n$, we have that
$$T_k(c_1,\dots,c_n)=(-1)^kB_k^{(n)}(x_1,\dots, x_n)/k!, \quad k\leq n$$
 (see \cite{Hirz-66} and formula (\ref{bern}) above).
  In particular, if $x_1=1$ and $x_i=0, \, i>1$ we have
 \beq{TB}
 T_k(1,0,0,\dots)=(-1)^k\frac{B_k}{k!},
 \eeq
where $B_j$ are the Bernoulli numbers determined by the generating function
$$
\frac{x}{e^x-1}=\sum_{n}B_n\frac{x^n}{n!}:
$$
$$B_0=1, B_2=\frac{1}{6}, B_4 = - \frac{1}{30}, B_6 =  \frac{1}{42}, B_8 = -\frac{1}{30}, B_{10} =  \frac{5}{66}, B_{12}= -\frac{691}{2730}$$
(all $B_k$ with odd $k$ are zero, except $B_1=-\frac{1}{2}$).

A natural question is to look at their denominators.
In 1840 von Staudt \cite{Staudt} and Clausen \cite{Clausen} independently found a remarkable result that
$$
B_{2n}+\sum_{p-1| 2n}\frac{1}{p}\in \mathbb Z,
$$
where the sum is taken over all primes $p$, such $p-1$ is a divisor of $2n.$ 
This means that the denominator $q_{2n}$ of the Bernoulli number $B_{2n}=\frac{p_{2n}}{q_{2n}}$ can be given as the product of these primes:
$$
q_{2n}=\prod_{p-1| 2n} p.
$$
This is reminiscent of the Hirzebruch formula (although Hirzebruch seems to be not aware of this link at the time).

The proof of von Staudt-Clausen formula (due to von Staudt) is based on the formula for the Bernoulli number
\beq{vSC}
B_{m}=\sum_{k=0}^{m} \frac{(-1)^k k! S(m,k)}{k+1},
\eeq
where $S(m,k)$ are the Stirling numbers of second kind \cite{Stanley}. We should mention also that there are several other explicit formulas for the Bernoulli numbers as the sums of fractions (see \cite{Gould}).


As a corollary of our results we have one of such representations due to Jordan (see \cite{Jordan}, p. 247, formula (5)), see also \cite{Shirai_2001},\cite{Woon}. \footnote{We are grateful to Alexey Ustinov and to Christophe Vignat for providing these references.}
 
 \begin{Theorem} Bernoulli number $B_k$ can be represented as the following sum over all partitions $\lambda=(\lambda_1,\dots, \lambda_l)$ of $k$:
 \beq{Bernnew}
B_k=k!\sum_{\lambda: |\lambda|=k}\frac{(-1)^{l}N(\lambda)}{(\lambda_1+1)!\dots (\lambda_l+1)!},
\eeq
where $N(\lambda)$ is the number of all different permutations of $(\lambda_1,\dots, \lambda_l)$.
\end{Theorem}

\begin{proof}
By Theorem  1.2  $T_k$ can be expressed in terms of the forgotten symmetric functions as follows 
$$
T_k=\sum_{\lambda: |\lambda|=k}\frac{f_\lambda}{(\lambda_1+1)!\dots (\lambda_l+1)!},
$$
so due to (\ref{TB}) we need only to compute the values $f_\lambda(1,0,0,\dots).$
For this we use formula (\ref{comb}) expressing $f_\lambda$ via monomial symmetric functions $m_\mu$: 
$$
(-1)^{|\lambda|-l}f_\lambda=\sum_{\mu}a_{\lambda \mu}m_\mu,
$$
where $a_{\lambda \mu}$ equals the number of different permutations $\alpha=(\alpha_1,\dots,\alpha_l)$ of  $\lambda=(\lambda_1,\dots, \lambda_l)$, such that
\beq{mual}
 \{\mu_1+\dots + \mu_j: \,\, 1\leq j \leq l(\mu)\} \subseteq \{\alpha_1+\dots + \alpha_i: \, 1\leq i \leq l\}.
\eeq
Since $m_\mu (1,0,0,\dots)=0$ for all $\mu\neq (k)$ and $m_{(k)}(1,0,0,\dots)=1$, we need only to find the coefficient $a_{\lambda (k)}$.
From (\ref{mual}) we see that $a_{\lambda (k)}$  is just the number of all different permutations of $(\lambda_1,\dots, \lambda_l):$
$$
f_\lambda(1,0,0,\dots)=a_{\lambda (k)}=N(\lambda)=\frac{l(\lambda)!}{m_1(\lambda)!\dots m_l(\lambda)!},
$$
where $m_j(\lambda)$ is the number of appearances of $j$ in $(\lambda_1,\dots, \lambda_l),$ which leads to formula (\ref{Bernnew}).
\end{proof}

In particular, for $k\leq 4$ we have
$$
B_2=2! \left(\frac{-1}{3!}+\frac{1}{2!\cdot 2!}\right)=\frac{1}{6}, \,\,\quad
B_3=3!\left(-\frac{1}{4!}+\frac{2}{3!\cdot 2!}-\frac{1}{2!\cdot 2!\cdot 2!}\right)=0,\,\,
$$
$$
B_4=4!\left(-\frac{1}{5!}+\frac{2}{4!\cdot 2!}-\frac{1}{3!\cdot 3!} -\frac{3}{3!\cdot 2!\cdot 2!}+\frac{1}{2!\cdot 2!\cdot 2!\cdot 2!}\right)=-\frac{1}{30}.
$$
As a corollary we can claim that the denominators of the Bernoulli numbers divide the corresponding Hirzebruch numbers (which is of course much weaker than the von Staudt-Clausen result). 

Our formula (\ref{Toddnew}) provides another representation for the Bernoulli numbers as the sum over partitions, which seems to be new \footnote{Tom Copeland has recently informed us that this formula can be derived from the known results about partition polynomials, see his comments on MathOverflow \cite{Copeland}, where one can find also new  interesting relations.}:
 \beq{Bernnew2}
B_k=k!\sum_{\lambda: |\lambda|=k}\frac{G(\lambda)}{(\lambda_1+1)\dots (\lambda_l+1)}, \quad G(\lambda)=\sum_{\mu:|\mu|=k} C_\lambda^\mu N(\mu)
 \eeq
   and $C_\lambda^\mu$ are defined by (\ref{clm}). However, to make it completely explicit we need a good combinatorial formula for $G(\lambda)=g_{\lambda}(1,0,0,\dots)$, which is one more reason to study these new symmetric functions.

For small $k$ from (\ref{g2}), (\ref{g3}), (\ref{g4}) we have the following representations of the Bernoulli numbers
$$
B_2=2!\left(\frac{1}{3}-\frac{1}{2\cdot 2}\right)=\frac{1}{6}, \quad B_3=3! \left(-\frac{1}{4}+\frac{3}{3\cdot 2}-\frac{2}{2\cdot2\cdot 2}\right)=0,
$$
$$
B_4=4! \left(\frac{1}{5}-\frac{4}{4\cdot 2}-\frac{2}{3\cdot 3}+\frac{10}{2\cdot 2\cdot 3}-\frac{5}{2\cdot 2\cdot 2\cdot 2}\right)=-\frac{1}{30}.
$$

\section{Acknowledgements}

We are very grateful to Alexey Ustinov and Fedor Popelenskiy for useful discussions and to Alexander Gaifullin for the constructive critical comments.

Our special thanks go to Sergei Petrovich Novikov, who is one of the founders of the theory of complex cobordisms, which plays a key role in the proof of our results.
We appreciate very much his support and encouragement from the start of our scientific career.

\medskip

{\bf Note added in proof.} Sergei Petrovich Novikov had passed away on June 6, 2024. This paper is dedicated to his blessed memory, which will live in our hearts forever.


\begin{thebibliography}{99}


\bibitem{AH}
Atiyah, M.F. and Hirzebruch, F. {\it Cohomologie-Operationen und charakteristische Klassen.}
Math. Zeitschr. {\bf 77}, 149--187 (1961).

\bibitem{B-1969}
Buchstaber, V.M. {\it Modules of differentials of the Atiyah–Hirzebruch spectral sequence.} Math. USSR-Sb. {\bf 7:2} (1969), 299-313. 

\bibitem{B-1970}   
Buchstaber, V.M. {\it Chern-Dold character in cobordisms,~I.} Math. Sbornik {\bf 83} (1970), 575-95.

\bibitem{BMN-71}
Buchstaber, V.M., Mischenko, A.S., Novikov, S.P. {\it  Formal groups and their role in the apparatus of algebraic topology.} Russian Math. Surveys {\bf 26:2}, 63-90 (1971).

\bibitem{B-2012}   
Buchstaber V.M.: {\it Complex cobordisms and formal groups.} Russian Math. Surveys, {\bf 67:5}, 111--174 (2012).

\bibitem{BV-2020}
Buchstaber, V.M., Veselov A.P. {\it Chern-Dold character in complex cobordisms and theta divisors.} Adv. Math. {\bf 449} (2024), 109720 , 35 pp., arXiv: 2007.05782.

\bibitem{Clausen}
Clausen, T. {\it Theorem.} Astronomische Nachrichten, {\bf 17 (22)} (1840), 351-352.

\bibitem{Comtet}
Comtet, L. {\it Advanced Combinatorics.} D. Reidel Publ. Comp., Dordrecht, 1974.

\bibitem{Copeland}
Copeland T. 
https://mathoverflow.net/questions/412573/combinatorics-for- the-action-of-virasoro-kac-schwarz-operators-partition-poly


\bibitem{Gessel}
Gessel, I.M. {\it Lagrange inversion.} J. Combinatorial Theory, Series A {\bf 144} (2016), 212-249.

\bibitem{Gould}
Gould, H.W. {\it Explicit formulas for Bernoulli numbers.} Amer. Math. Monthly {\bf 79} (1972), 44-51.
    
%
\bibitem{Hirz-58}
Hirzebruch, F.: {\it Komplexe Mannigfaltigkeiten.} In: Proc. Intern. Congress of Math. 1958, 119-136. Camb. Univ. Press, 1960. 

\bibitem{Hirz-56}
Hirzebruch, F. {\it Neue Topologische Methoden in der Algebraischen Geometrie.} Ergebnisse der Mathematik und Ihrer Grenzgebiete. 1. Folge (MATHE1, volume N. F., 9). Springer, Berlin, 1956.

\bibitem{Hirz-66}
Hirzebruch, F. {\it Topological Methods in Algebraic Geometry.} Springer, 1966. 

\bibitem{Jordan}
Jordan, C. {\it Calculus of Finite Differences.} Budapest, 1939; Second ed., Chelsea, New York, 1950.

\bibitem{Lenart}
Lenart, C. {\it Symmetric functions, formal group laws, and Lazard's theorem.} Adv. Math. {\bf 134} (1998), 219--239.

\bibitem{Loday}
Loday, J.-L. {\it The multiple faces of the associahedron.} Clay Mathematics Institute Publication, 2005.


\bibitem{Mac}
Macdonald, I.G. {\it Symmetric functions and Hall polynomials.} 2nd
edition, Oxford Univ. Press (1995).

\bibitem{McMullen}
McMullen, C.T. {\it Moduli spaces in genus zero and inversion of power series.} L’Enseignement Math\'ematique {\bf (2) 60} (2014), 25–30. 

\bibitem{Mil-60}
Milnor, J.: {\it On the cobordism ring $\Omega_*$ and complex analogue.}
Part I., Amer. J. Math. {\bf 82:3}, 505--52 (1960).


\bibitem{MS}
Milnor, J., Stasheff J.D. {\it Characteristic Classes.}  Ann. Math. Studies {\bf 76} (1974).

 \bibitem{Nor}
 N\"orlund, N. E. {\it Differenzenrechnung.} Berlin, Springer-Verlag, 1924.
 
 \bibitem{Nov-60}   
Novikov, S.P. \emph{Some problems of topology of manifolds connected with the theory of Thom spaces.} DAN SSSR {\bf 132:5} (1960), 1031-1034.

\bibitem{Nov-62}   
Novikov, S.P. \emph{Homotopy properties of Thom complexes.} Math. Sbornik {\bf 57(99):4}, 407--442 (1962).

\bibitem{Nov-67}   
Novikov, S.P.  {\it Methods of algebraic topology from the viewpoint of cobordism theory.} 
Izv.  Akad.  Nauk  SSSR  (Math. USSR - Izvestija) Ser.  Mat.  {\bf 31:4},  827--913 (1967).

\bibitem{Rozh}
Rozhdestvenskii, V.  {\it A lower bound in the problem of realization of cycles.} arXiv:2303.10240, 2023.  Journal of Topology {\bf 16:4} (2023), 1475-1508.

\bibitem{Shirai_2001}
Shirai, S., Sato, K. {\it Some identities involving Bernoulli and Stirling numbers.}
J. Number Theory {\bf 90} (2001), 130-142.


\bibitem{Stanley} 
Stanley, R.P. {\it Enumerative Combinatorics.} Vol. 2. Cambridge University Press, 1999.

\bibitem{Thom}
Thom, R. {\it Quelques proprietes globales des varietes differentiables.} Comm. Math. Helv.,
{\bf 28} (1954), 17-86.

\bibitem{Todd}
Todd, J.A. {\it The arithmetical invariants of algebraic loci.} Proc. of the London Mathematical Society, {\bf 43:1} (1938), 190-225.

\bibitem{Staudt}
Von Staudt, Ch. {\it Beweis eines Lehrsatzes, die Bernoullischen Zahlen betreffend.} J. Reine und Angewandte Mathematik {\bf 21} (1840), 372-374.

\bibitem{Woon}
Woon, S.C. {\it A tree for generating Bernoulli numbers.} Math. Magazine {\bf 70:1} (1997), 51-56.

\end{thebibliography}
\end{document}